\documentclass[12pt]{amsart}
\usepackage{latexsym}
\usepackage{amssymb,amsfonts}
\usepackage[colorlinks]{hyperref}

\newtheorem{theorem}{Theorem}[section]
\newtheorem{lemma}[theorem]{Lemma}
\newtheorem{proposition}[theorem]{Proposition}
\newtheorem{corollary}[theorem]{Corollary}

\newtheorem*{theoremA}{Theorem A}
\newtheorem*{theoremB}{Theorem B}
\newtheorem*{theoremC}{Theorem C}

\theoremstyle{definition}

\newtheorem*{notation}{Notation}

\theoremstyle{remark}

\newcommand{\PSL}{{\mathrm {PSL}}}

\newcommand{\PSU}{{\mathrm {PSU}}}

\newcommand{\Sp}{{\mathrm {Sp}}}
\newcommand{\PSp}{{\mathrm {PSp}}}

\newcommand{\Schur}{{\mathrm {Schur}}}

\newcommand{\Mult}{{\mathrm {Mult}}}

\newcommand{\Irr}{{\mathrm {Irr}}}

\newcommand{\St}{{\mathrm {St}}}

\newcommand{\CC}{{\mathbb C}}

\newcommand{\NN}{{\mathbb N}}

\newcommand{\ta}{\hspace{0.5mm}^{2}\hspace*{-0.2mm}}

\newcommand{\AAA}{ \hat{\textup{\textsf{A}}}_n}
\newcommand{\SSS}{\hat{\textup{\textsf{S}}}_n}
\newcommand{\Al}{\textup{\textsf{A}}}
\newcommand{\Sy}{\textup{\textsf{S}}}
\newcommand{\cd}{{\mathrm {cd}}}

\newcommand{\bC}{{\mathbf{C}}}

\newcommand{\bZ}{{\mathbf{Z}}}

\newcommand{\la}{\lambda}
\newcommand{\fch}{\operatorname{fch}}
\newcommand{\frb}{\operatorname{frb}}

\begin{document}

\title[The double covers of the symmetric and alternating groups]
{Complex group algebras of the double covers\\ of the symmetric and
alternating groups}

\author[C.\ Bessenrodt]{Christine Bessenrodt}
\address{Institut f\"ur Algebra, Zahlentheorie und Diskrete Mathematik,
Leibniz Universit\"at Hannover,
D-30167 Hannover, Germany} \email{bessen@math.uni-hannover.de}

\author[H.\,N. Nguyen]{Hung Ngoc Nguyen}
\address{Department of Mathematics, The University of Akron, Akron,
Ohio 44325, United States} \email{hungnguyen@uakron.edu}

\author[J.\,B. Olsson]{J{\o}rn B. Olsson}
\address{Department of Mathematical Sciences, University of Copenhagen, DK-2100 Copenhagen \O,
Denmark} \email{olsson@math.ku.dk}

\author[H.\,P. Tong-Viet]{Hung P. Tong-Viet}
\address{Department of Mathematics and Applied Mathematics,
University of Pretoria,
Private Bag X20, Hatfield, Pretoria 0028,
South Africa}
\email{Hung.Tong-Viet@up.ac.za}

\thanks{H.\,N.~Nguyen was partially supported by the NSA Young
  Investigator Grant \#H98230-14-1-0293 and a Faculty Scholarship
  Award from Buchtel College of Arts and Sciences, The University
  of Akron. Tong-Viet's work is based on the research supported in part by the National Research Foundation of South Africa (Grant Number 93408).}

\subjclass[2010]{Primary 20C30, 20C15; Secondary 20C33}

\keywords{Symmetric groups, alternating groups, complex group
algebras, Schur covers, double covers, irreducible representations,
character degrees}

\date{February 24, 2015}

\begin{abstract} We prove that the double covers of the alternating and symmetric
groups are determined by their complex group algebras. To be more
precise, let $n\geq 5$ be an integer, $G$ a finite group, and let
$\AAA$ and $\SSS^\pm$ denote the double covers of $\Al_n$ and
$\Sy_n$, respectively. We prove that $\CC G\cong \CC \AAA$ if and
only if $G\cong \AAA$, and $\CC G\cong \CC \SSS^+\cong\CC\SSS^-$ if
and only if $G\cong \SSS^+$ or $\SSS^-$. This in particular
completes the proof of a conjecture proposed by the second and
fourth authors that every finite quasi-simple group is determined uniquely
up to isomorphism by the structure of its complex group algebra. The
known results on prime power degrees and relatively small degrees of
irreducible (linear and projective) representations of the symmetric
and alternating groups together with the classification of finite
simple groups play an essential role in the proofs.
\end{abstract}

\maketitle


\section{Introduction}
The complex group algebra of a finite group $G$, denoted by $\CC G$,
is the set of formal sums $\{\sum_{g\in G}a_gg\mid a_g\in \CC\}$
equipped with natural rules for addition, multiplication, and scalar
multiplication. Wedderburn's theorem implies that $\CC G$ is
isomorphic to the direct sum of matrix algebras over $\CC$ whose
dimensions are exactly the degrees of the (non-isomorphic)
irreducible complex representations of $G$. Therefore, the study of
complex group algebras and the relation to their base groups is
important in group representation theory.

In an attempt to understand the connection between the structure of
a finite group and its complex group algebra, R.~Brauer asked in his
landmark paper~\cite[Question~2]{Brauer}: \emph{when do
non-isomorphic groups have isomorphic complex group algebras}? Since
this question might be too general to be solved completely, it is
more feasible to study more explicit questions/problems whose
solutions will provide a partial answer to Brauer's question. For
instance, if two finite groups have isomorphic complex group
algebras and one of them is solvable, is it true that the other is
also solvable? Or, if two finite groups have isomorphic complex
group algebras and one of them has a normal Sylow $p$-subgroup, can
we conclude the same with the other group? We refer the reader to
Brauer's paper~\cite{Brauer} or Section~9 of the survey
paper~\cite{Navarro} by G.~Navarro for more discussions on complex
group algebras.

A natural problem that arises from Brauer's question is the
following: given a finite group $G$, determine all finite groups (up
to isomorphism) such that their complex group algebra is isomorphic to that of~$G$.
This problem is easy for abelian groups but difficult for
solvable groups in general. If $G$ is any finite abelian group of
order $n$, then $\CC G$ is isomorphic to a direct sum of $n$ copies
of $\CC$, so that the complex group algebras of any two abelian
groups are isomorphic if and only if the two groups have the same
order. For solvable groups, the probability that two groups have
isomorphic complex group algebra is often fairly `high'. For
instance, B.~Huppert pointed out in~\cite{Huppert1} that among 2328
groups of order~$2^7$, there are only 30 different structures of
complex group algebras. In contrast to solvable groups, simple
groups or more generally quasi-simple groups seem to have a stronger
connection with their complex group algebras.
In~\cite{Nguyen-TongViet}, two of the four authors have conjectured
that every finite quasi-simple group is determined uniquely up to
isomorphism by its complex group algebra, and proved it for all
quasi-simple groups, except the non-trivial perfect central covers
of the alternating groups.

Let $\Al_n$ denote the alternating group of degree $n$. (Throughout
the paper we always assume that $n\geq5$ unless otherwise stated.)
The \emph{Schur covers} (or \emph{covering groups}) of the
alternating groups as well as symmetric groups were first studied
and classified by I.~Schur~\cite{Schur} in connection with their
projective representations. It is known that $\Al_n$ has one
isomorphism class of Schur covers, which is indeed the double cover
$\AAA$ except when $n$ is 6 or 7, where triple and $6$-fold covers
also exist. We are able to prove that every double cover of an
alternating group of degree at least $5$ is determined uniquely by
its complex group algebra.

\begin{theoremA}\label{theorem_alternating} Let $n\geq5$. Let $G$ be a finite group
and $\AAA$ the double cover of $\Al_n$. Then $G\cong \AAA$ if and
only if $\CC G\cong \CC\AAA$.
\end{theoremA}

We prove a similar result for the triple and $6$-fold (perfect
central) covers of $\Al_6$ and~$\Al_7$, and therefore complete the
proof of the aforementioned conjecture.

\begin{theoremB}\label{conjecture} Let $G$ be a finite group
and $H$ a quasi-simple group. Then $G\cong H$ if and only if $\CC
G\cong \CC H$.
\end{theoremB}

\begin{proof} This is a consequence of Theorem~\href{theorem_alternating}{A}, Theorem~\ref{theorem-A6andA7},
and~\cite[Corollary~1.4]{Nguyen-TongViet}.
\end{proof}

The symmetric group $\Sy_n$ has two isomorphism classes of Schur
double covers, denoted by $\SSS^-$ and $\SSS^+$. It turns out that
these two covers are isoclinic and therefore their complex group
algebras $\CC \SSS^+$ and $\CC\SSS^-$ are isomorphic \cite{Morris}.
Our next result solves the above problem for the double covers of
the symmetric groups.

\begin{theoremC}\label{theorem_symmetric} Let $n\geq5$. Let $G$ be a finite group
and $\SSS^\pm$ the double covers of $\Sy_n$. Then $\CC G\cong \CC
\SSS^+$ \textup{(}or equivalently $\CC G\cong \CC\SSS^-$\textup{)}
if and only if $G\cong \SSS^+$ or $G\cong\SSS^-$.
\end{theoremC}

Let $\Irr(G)$ denote the set of all irreducible representations (or
characters) of a group $G$ over the complex field. As mentioned
above, two finite groups have isomorphic complex group algebras if
and only if they have the same set of degrees (counting
multiplicities) of irreducible characters. Therefore, the proofs of
our main results as expected depend heavily on the representation
theories of the symmetric and alternating groups, their double
covers, and quasi-simple groups in general. In particular, we make
use of the known results on relatively small degrees and prime power
degrees of the irreducible characters of $\AAA$ and $\SSS^\pm$.

The remainder of the paper is organized as follows. In the next
section, we give a brief overview of the representation theory of
the symmetric and alternating groups as well as their double covers,
and then collect some results on prime power character degrees of
these groups. The results on minimal degrees are then presented in
Section~\ref{section-low degree of groups}. In
Section~\ref{section-useful lemmas}, we establish some useful lemmas
that will be needed later in the proofs of the main results. The
proof of Theorem~\href{theorem_alternating}{A} is carried out in
Section~\ref{section-theoremA} and exceptional covers of $\Al_6$ and
$\Al_7$ are treated in Section~\ref{sectionA6andA7}.

The last four sections are devoted to the proof of
Theorem~\href{theorem_symmetric}{C}. Let $G$ be a finite group such
that $\CC G\cong \CC\SSS^\pm$. We will show that $G'=G''$ and
therefore there exists a normal subgroup $M$ of $G$ such that
$M\subseteq G'$ and the chief factor $G'/M$ is isomorphic to a
direct product of $k$ copies of a non-abelian simple group $S$. To
prove that $G$ is isomorphic to one of $\SSS^\pm$, one of the key
steps is to show that this chief factor is isomorphic to $\Al_n$. We
will do this by using the classification of finite simple groups to
eliminate almost all possibilities for $k$ and $S$. As we will see,
it turns out that the case of simple groups of Lie type in even
characteristic is most difficult.

In Section~\ref{section:nondegree}, we prove a non-existence result
for particular character degrees of $\SSS^\pm$ and apply it in
Section~\ref{section eliminating even charactristic} to show that
$S$ cannot be a simple group of Lie type in even characteristic. We
then eliminate other possibilities for $S$ in Section~\ref{section
eliminating simple groups} and complete the proof of
Theorem~\href{theorem_symmetric}{C} in Section~\ref{section theorem
C}.

We close this introduction with many thanks to the referee for
his/her careful reading of the manuscript and several helpful
comments.

\begin{notation} Since $\CC \SSS^+\cong\CC \SSS^-$, when working with character
degrees of $\SSS^\pm$, it suffices to consider just one of the two
covers. For the sake of convenience, we will write $\SSS$ to denote
either one of the two double covers of $\Sy_n$. If $X$ and $Y$ are
two multisets, we write $X\subseteq Y$ and say that $X$ is a
sub-multiset of $Y$ if the multiplicity of any element in $X$ does
not exceed that of the same element in $Y$. For a finite group $G$,
the number of conjugacy classes of $G$ is denoted by $k(G)$. We
write $\Irr_{2'}(G)$ to mean the set of all irreducible characters
of $G$ of odd degree. The set and the multiset of character degrees
of $G$ are denoted respectively by $\cd(G)$ and $\cd^\ast(G)$.
Finally, we denote by $d_i(G)$ the $i$th smallest non-trivial
character degree of $G$. Other notation is standard or will be
defined when needed.
\end{notation}


\section{Prime power degrees of $\SSS$ and $\AAA$}\label{section-degrees of Sn and An}

In this section, we collect some results on irreducible characters
of prime power degree of $\SSS$ and $\AAA$. The irreducible
characters of the double covers of the symmetric and alternating
groups are divided into two kinds: faithful characters which are
also known as spin characters, and non-faithful characters which can
be viewed as ordinary characters of $\Sy_n$ or $\Al_n$.

\subsection{Characters of $\Sy_n$ and $\Al_n$} To begin with, let us recall some notation and terminology of
partitions and Young diagrams in connection with representation
theory of the symmetric and alternating groups. A partition
$\lambda$ of $n$ is a finite sequence of natural numbers
$(\lambda_1,\lambda_2,...)$ such that
$\lambda_1\geq\lambda_2\geq\cdot\cdot\cdot$ and
$\lambda_1+\lambda_2+\cdot\cdot\cdot=n$. If
$\lambda_1>\lambda_2>\cdots$, we say that $\lambda$ is a \emph{bar
partition} of $n$ (also called a \emph{strict partition} of~$n$).
The Young diagram associated to $\lambda$ is an array of $n$ nodes
with $\lambda_i$ nodes on the $i$th row. At each node $(i,j)$, we
define the \emph{hook length} $h(i,j)$ to be the number of nodes to
the right and below the node $(i,j)$, including the node $(i,j)$.

It is well known that the irreducible characters of $\Sy_n$ are in
one-to-one correspondence with partitions of $n$. The degree of the
character $\chi_\lambda$ corresponding to $\lambda$ is given by the
following \emph{hook-length formula} of Frame, Robinson and
Thrall~\cite{Frame-Robinson-Thrall}:
$$f_\lambda:=\chi_{\lambda}(1)=\frac{n!}{\prod _{i,j}h(i,j)} \:.$$

Two partitions of $n$ whose Young diagrams transform into each other
when reflected about the line $y=-x$, with the coordinates of the
upper left node taken as $(0, 0)$, are called conjugate partitions.
The partition conjugate to $\lambda$ is denoted by
$\overline{\lambda}$. If $\lambda=\overline{\lambda}$, we say that
$\lambda$ is self-conjugate. The irreducible characters of $\Al_n$
can be obtained by restricting those of $\Sy_n$ to $\Al_n$. More
explicitly,
$\chi_{\lambda}\downarrow_{\Al_n}=\chi_{\overline{\lambda}}\downarrow_{\Al_n}$
is irreducible if $\lambda$ is not self-conjugate. Otherwise,
$\chi_{\lambda}\downarrow_{\Al_n}$ is the sum of two different
irreducible characters of $\Al_n$ of the same degree. In short, the
degrees of irreducible characters of $\Al_n$ are labeled by
partitions of $n$ and are given by
$$ \widetilde{f}_\lambda= \left\{\begin
{array}{ll}
f_\lambda & \text{ if } \lambda\neq\overline{\lambda},\\
f_\lambda/2 & \text{ if } \lambda=\overline{\lambda}.
\end {array} \right.$$

Irreducible representations of prime-power degree of the symmetric
and alternating groups were classified by A.~Balog, C.~Bessenrodt,
J.\,B.~Olsson and K.~Ono, see~\cite{Balog-Bessenrodt-Olsson-Ono}.
This result is critical in eliminating simple groups other than
$\Al_n$ involved in the structure of finite groups whose complex
group algebras are isomorphic to $\CC\AAA$ or~$\CC\SSS$.

\begin{lemma}[\cite{Balog-Bessenrodt-Olsson-Ono}, Theorem 2.4]\label{prime power degrees of Sn} Let $n\geq 5$. An
irreducible character $\chi_\lambda\in\Irr(\Sy_n)$ corresponding to
a partition $\lambda$ of $n$ has prime power degree
$f_\lambda=p^r>1$ if and only if one of the following occurs:
\begin{enumerate}
\item $n=p^r+1$, $\lambda=(n-1,1)$ or $(2,1^{n-2})$, and ${f}_\lambda=n-1$;
\item $n=5$, $\lambda=(2^2,1)$ or $(3,2)$, and $f_\lambda=5$;
\item $n=6$, $\lambda=(4,2)$ or $(2^2,1^2)$, and $f_\lambda=9$; $\lambda=(3^2)$ or $(2^3)$, and
$f_\lambda=5$; $\lambda=(3,2,1)$ and $f_\lambda=16$;
\item $n=8$, $\lambda=(5,2,1)$ or $(3,2,1^3)$, and ${f}_\lambda=64$;
\item $n=9$, $\lambda=(7,2)$ or $(2^2,1^5)$, and ${f}_\lambda=27$.
\end{enumerate}
\end{lemma}

\begin{lemma}[\cite{Balog-Bessenrodt-Olsson-Ono}, Theorem 5.1]\label{prime power degrees of An} Let $n\geq 5$. An
irreducible character degree $\widetilde{f}_\lambda$ of $\Al_n$
corresponding to a partition $\lambda$ of $n$
is a prime power $p^r>1$ if and only if one of the following occurs:
\begin{enumerate}
\item $n=p^r+1$, $\lambda=(n-1,1)$ or $(2,1^{n-2})$, and $\widetilde{f}_\lambda=n-1$;
\item $n=5$, $\lambda=(2^2,1)$ or $(3,2)$, and
$\widetilde{f}_\lambda=5$; $\lambda=(3,1^2)$ and
$\widetilde{f}_\lambda=3$;
\item $n=6$, $\lambda=(4,2)$ or $(2^2,1^2)$ and
$\widetilde{f}_\lambda=9$; $\lambda=(3^2)$ or $(2^3)$ and
$\widetilde{f}_\lambda=5$; $\lambda=(3,2,1)$ and
$\widetilde{f}_\lambda=8$;
\item $n=8$, $\lambda=(5,2,1)$ or $(3,2,1^3)$, and $\widetilde{f}_\lambda=64$;
\item $n=9$, $\lambda=(7,2)$ or $(2^2,1^5)$, and $\widetilde{f}_\lambda=27$.
\end{enumerate}
\end{lemma}


\subsection{Spin characters of $\Sy_n$ and $\Al_n$} We now recall the spin representation theory of the
symmetric and alternating groups, due to
Schur~\cite{Hoffman-Humphrey,Morris,Schur,Wagner}. To each bar partition
$\mu=(\mu_1,\mu_2,...,\mu_m)$ (i.e., $\mu_1>\mu_2> \cdots>\mu_m$) of
$n$, there corresponds one or two irreducible characters (depending
on whether $n-m$ is even or odd, respectively) of $\SSS$ of degree
$$g_\mu=2^{\lfloor (n-m)/2\rfloor} \bar g_\mu$$
where $\bar g_\mu$ denotes the number of shifted standard tableaux
of shape~$\mu$. This number can be computed by an analogue of the
hook-length formula, the \emph{bar formula}
\cite[Prop.~10.6]{Hoffman-Humphrey}. The length $b(i,j)$ of the
$(i,j)$-bar of $\mu$ is the length of the $(i,j+1)$-hook in the
shift-symmetric diagram of $\mu$ (obtained by reflecting the shifted
diagram of $\mu$ along the diagonal and pasting it onto $\mu$; see
\cite[p.14]{Macdonaldbook} for details). Then
\[\bar g_\mu =\frac{n!}{\prod_{i,j} b(i,j)}\:.\] The spin character
degree may also be computed by the formula
\[g_\mu=2^{\lfloor
(n-m)/2\rfloor}\frac{n!}{\mu_1!\mu_2!\cdots\mu_m!}\prod_{i<j}
\frac{\mu_i-\mu_j}{\mu_i+\mu_j}.\]

Again, one can get faithful irreducible characters of $\AAA$ by
restricting those of $\SSS^\pm$ to $\AAA$ in the following way. If
$n-m$ is odd, then the restrictions of the two characters of
$\SSS^\pm$ labeled by $\mu$ to $\AAA$ are the same and irreducible.
Otherwise, the restriction of the one character labeled by $\mu$ is
the sum of two irreducible characters of the same degree $g_\mu/2$.
Let $\widetilde{g}_\mu$ be the degree of the irreducible spin
character(s) of $\AAA$ labeled by the bar partition $\mu$; we then
have
$$ \widetilde{g}_\mu= \left\{\begin
{array}{ll}
g_\mu & \text{ if } n-m \text{ is odd},\\
g_\mu/2 & \text{ if } n-m \text{ is even}.
\end {array} \right.$$

The classification of spin representations of prime power degree of
the symmetric and alternating groups has been done by the first and
third authors of this paper in~\cite{Bessenrodt-Olsson}.

\begin{lemma}[\cite{Bessenrodt-Olsson}, Theorem~4.2]\label{prime power degrees of Schur cover of Sn} Let
$n\geq5$ and $\mu$ be a bar partition of $n$. The spin irreducible
character degree $g_\mu$ of $\SSS$ corresponding to $\mu$
is a prime power if and only if one of the following occurs:
\begin{enumerate}
\item $\mu=(n)$ and $g_\mu=2^{\lfloor(n-1)/2\rfloor}$;
\item $n=2^r+2$ for some $r\in\NN$, $\mu=(n-1,1)$, and $g_\mu=2^{2^{r-1}+r}$;
\item $n=5$, $\mu=(3,2)$, and $g_\mu=4$;
\item $n=6$, $\mu=(3,2,1)$, and $g_\mu=4$;
\item $n=8$, $\mu=(5,2,1)$, and $g_\mu=64$.
\end{enumerate}
\end{lemma}

\begin{lemma}[\cite{Bessenrodt-Olsson}, Theorem~4.3]\label{prime power degrees of Schur cover of An} Let
$n\geq5$ and $\mu$ be a bar partition of $n$. The spin irreducible
character degree $\widetilde{g}_\mu$ of $\AAA$ corresponding to
$\mu$
is a prime power if and only if one of the following
occurs:
\begin{enumerate}
\item $\mu=(n)$ and $\widetilde{g}_\mu=2^{\lfloor(n-2)/2\rfloor}$;
\item $n=2^r+2$ for some $r\in\NN$, $\mu=(n-1,1)$, and $\widetilde{g}_\mu=2^{2^{r-1}+r-1}$;
\item $n=5$, $\mu=(3,2)$, and $\widetilde{g}_\mu=4$;
\item $n=6$, $\mu=(3,2,1)$, and $\widetilde{g}_\mu=4$;
\item $n=8$, $\mu=(5,2,1)$, and $\widetilde{g}_\mu=64$.
\end{enumerate}
\end{lemma}


\section{Low degrees of $\SSS$ and $\AAA$}\label{section-low degree of groups}

We present in this section some results on minimal degrees of both
ordinary and spin characters of the symmetric and alternating
groups. We start with ordinary characters.

\begin{lemma}[Rasala~\cite{Rasala}]\label{minimal degrees of Sn} The following hold:
\begin{enumerate}
\item $d_1(\Sy_n)=n-1$ if $n\geq5$;

\item $d_2(\Sy_n)=n(n-3)/2$ if $n\geq9$;

\item $d_3(\Sy_n)=(n-1)(n-2)/2$ if $n\geq9$;

\item $d_4(\Sy_n)=n(n-1)(n-5)/6$ if $n\geq13$;

\item $d_5(\Sy_n)=(n-1)(n-2)(n-3)/6$ if $n\geq13$;

\item $d_6(\Sy_n)=n(n-2)(n-4)/3$ if $n\geq15$;

\item $d_7(\Sy_n)=n(n-1)(n-2)(n-7)/24$ if $n\geq15$.
\end{enumerate}
\end{lemma}

\begin{lemma}[Tong-Viet~\cite{Tong-Viet2}]\label{minimal degrees of An} If $n\geq 15$ then $d_i(\Al_n)=d_i(\Sy_n)$ for $1\leq i\leq
4$ and, if $n\geq 22$, then $d_i(\Al_n)=d_i(\Sy_n)$ for $1\leq i\leq
7$.
\end{lemma}

The minimal degrees of spin irreducible representations of $\AAA$
and $\SSS$ were obtained by A.~Kleshchev and P.\,H.~Tiep,
see~\cite{Kleshchev-Tiep,Kleshchev-Tiep-pre}. These minimal degrees
are indeed the degrees of the \emph{basic spin} and \emph{second
basic spin} representations. Let $\mathfrak{d}_1(\AAA)$ and
$\mathfrak{d}_1(\SSS)$ denote the smallest degrees of irreducible
spin characters of $\AAA$ and $\SSS$, respectively.

\begin{lemma}[\cite{Kleshchev-Tiep},
Theorem~A]\label{kleshchev-tiep} Let $n\geq8$. The smallest degree
of the irreducible spin characters of $\AAA$ ($\SSS$, resp.) is
$\mathfrak{d}_1(\AAA)=2^{\lfloor(n-2)/2\rfloor}$
($\mathfrak{d}_1(\SSS)=2^{\lfloor(n-1)/2\rfloor}$, resp.) and there
is no degree between $\mathfrak{d}_1(\AAA)$ ($\mathfrak{d}_1(\SSS)$,
resp.) and $2\mathfrak{d}_1(\AAA)$ ($2\mathfrak{d}_1(\SSS)$, resp.).
\end{lemma}

Using the above results, we easily deduce the following.

\begin{lemma}\label{minimal degrees of Schur covers} The following hold:
\begin{enumerate}
\item If $n\geq 31$, then $d_i(\SSS)=d_i(\Sy_n)$ for $1\leq i\leq 7$.
\item If $n\geq 34$, then $d_i(\AAA)=d_i(\Al_n)$ for $1\leq i\leq 7$.
\end{enumerate}
\end{lemma}

\begin{proof} We observe that
$\mathfrak{d}_1(\SSS)=2^{\lfloor(n-1)/2\rfloor}>n(n-1)(n-2)(n-7)/24=d_7(\Sy_n)$
if $n\geq 31$. Therefore part~(1) follows by Lemmas~\ref{minimal
degrees of Sn} and~\ref{kleshchev-tiep}. Similarly, we have
$\mathfrak{d}_1(\AAA)=2^{\lfloor(n-2)/2\rfloor}>n(n-1)(n-2)(n-7)/24=d_7(\Al_n)$
if $n\geq 34$ and thus part~(2) follows.
\end{proof}

\begin{lemma}\label{minimal degrees of Schur cover of An} Let $G$ be
either $\AAA$ or $\SSS$. Then we have:
\begin{enumerate}
\item If $n\geq8$, then $d_1(G)=n-1$.
\item If $n\geq10$, then
$d_2(G)=\min\{n(n-3)/2,\mathfrak{d}_1(G)\}$. Furthermore, if
$n\geq12$, then $d_2(G)>2n$.
\item If $n\geq16$, then $d_3(G)=(n-1)(n-2)/2$ and
$d_4(G)=\min\{n(n-1)(n-5)/6,\mathfrak{d}_1(G)\}$.
\item If $n\geq28$, then $d_4(G)=n(n-1)(n-5)/6$, $d_5(G)=(n-1)(n-2)(n-3)/6$, $d_6(G)=n(n-2)(n-4)/3$, and
$d_7(G)=\min\{n(n-1)(n-2)(n-7)/24,\mathfrak{d}_1(G)\}$.
\end{enumerate}
\end{lemma}

\begin{proof} When $n\geq 8$, we see that
$\mathfrak{d}_1(\SSS)\geq\mathfrak{d}_1(\AAA)=2^{\lfloor(n-2)/2\rfloor}>n-1$
and $d_1(\Sy_n)=d_1(\Al_n)=n-1$, which imply that
$d_1(\AAA)=d_1(\SSS)=n-1$ and part~(1) follows. When $n\geq 10$, we
observe that $d_2(\Al_n)=d_2(\Sy_n)=n(n-3)/2$ and part~(2) then
follows from part~(1).

Suppose that $n\geq 16$. It is easy to check that
$2^{\lfloor(n-2)/2\rfloor} >(n-1)(n-2)/2=d_3(\Al_n)=d_3(\Sy_n)$. It
follows that $d_3(\AAA)=d_3(\SSS)=(n-1)(n-2)/2$ and
$d_4(G)=\min\{n(n-1)(n-5)/6,\mathfrak{d}_1(G)\}$, as claimed.

Finally suppose that $n\geq 28$. We check that
$2^{\lfloor(n-2)/2\rfloor}> n(n-2)(n-4)/3=d_6(\Al_n)=d_6(\Sy_n)$ and
part~(4) then follows.
\end{proof}


\section{Some useful lemmas}\label{section-useful lemmas}

We begin with an easy observation.

\begin{lemma}\label{prime number between n and d2} There always exists
a prime number $p$ with $n<p\leq d_2(G)$ for $G=\AAA$ or $\SSS$,
provided that $n\geq 9$.
\end{lemma}

\begin{proof} It is routine to check the statement for $9\leq n\leq
11$ by using \cite{Atl1}. So we can assume that $n\geq12$. By
Lemma~\ref{minimal degrees of Schur cover of An}, it suffices to
prove that there exists a prime between $n$ and $2n$. However, this
is the well-known Bertrand-Chebyshev
theorem~\cite{Harborth-Kemnitz}.
\end{proof}

The next three lemmas are critical in the proof of
Theorem~\href{theorem_alternating}{A}.

\begin{lemma}\label{number of odd degree characters} Let $S$ be a
simple group of Lie type of rank $l$ defined over a field of $q$
elements with $q$ even. Then
\[|\Irr_{2'}(S)|\geq \left\{\begin {array}{ll}
q^{l}/(l+1,q-1) & \text{ if } S \text{ is of type } A,\\
q^{l}/(l+1,q+1) & \text{ if } S \text{ is of type } \ta A,\\
q^{l}/3 & \text{ if } S \text{ is of type } E_6,\ta E_6,\\
q^l & \text{ otherwise. }
\end {array} \right.\]
\end{lemma}

\begin{proof} Let $S_{sc}$ be the finite Lie type group of
simply-connected type corresponding to~$S$.
By~\cite[Corollary~3.6]{Brunat}, $S_{sc}$ has $q^l$ semisimple
conjugacy classes. To each semisimple class $s$ of $S_{sc}$,
Lusztig's classification of complex characters of finite groups of
Lie type says that there corresponds a semisimple character of the
dual group, say $S_{sc}^\ast$, of $S_{sc}$ of degree
$$|S_{sc}|_{2'}/|\bC_{S_{sc}}(s)|_{2'}\:.$$
This means that the dual $S_{sc}^\ast$ of $S_{sc}$ has at least
$q^l$ irreducible characters of odd degree.

If $S$ is of type $A$, we have
$S^\ast_{sc}=\mathrm{PGL}_{l+1}(q)=S.(l+1,q-1)$ and the lemma
follows for linear groups. A similar argument works for unitary
groups and $E_6$ as well as $\ta E_6$. If $S$ is not of these types,
we will have $S=S_{sc}=S_{sc}^\ast$ and the lemma also follows.
\end{proof}

\begin{lemma}\label{number of odd degree characters of Sn} Let
$n\geq5$ and let $n=2^{k_1}+2^{k_2}+\cdots+2^{k_t}$ be the binary
expansion of~$n$ with $k_1>k_2>\cdots>k_t\geq0$. Then
$$|\Irr_{2'}(\AAA)|\leq|\Irr_{2'}(\SSS)|= 2^{k_1+k_2+\cdots+k_t} \:.$$
\end{lemma}

\begin{proof} If $\mu=(\mu_1,\mu_2,...,\mu_m)$ is a bar partition of
$n$, then the $2$-part of the spin character degree $g_\mu$ of
$\SSS$ labeled by $\mu$ is at least
$$2^{\lfloor(n-m)/2\rfloor},$$ which is at least
$2$ as $n\geq5$ and $n-m\geq 3$. We note that if $n-m=3$ then
$\widetilde{g}_\mu=g_\mu$. Therefore, the $2$-part of the degree
$\widetilde{g}_\mu$ of $\AAA$ is at least $2$ as well. In
particular, we see that every spin character degree of $\AAA$ as
well as $\SSS$ is even. It follows that
$$|\Irr_{2'}(\AAA)|=|\Irr_{2'}(\Al_n)| \text{ and } |\Irr_{2'}(\SSS)|=|\Irr_{2'}(\Sy_n)| \:.$$

As mentioned in~\cite{Mckay}, the number of odd degree irreducible
characters of $\Al_n$ does not exceed that of $\Sy_n$. Now the lemma
follows from the formula for the number of odd degree characters of
$\Sy_n$ given in~\cite[Corollary~1.3]{Macdonald}.
\end{proof}

\begin{lemma}\label{technical lemma for PSL and PSU -Theorem A} Let
$n=2^{k_1}+2^{k_2}+\cdots+2^{k_t}$ be the binary expansion of $n$
with $k_1>k_2>\cdots>k_t\geq0$.
\begin{enumerate}
\item If $k_1+k_2+\cdots+k_t\geq\sqrt{(n-3)/2}$, then
$n<2^{15}$;
\item if $k_1+k_2+\cdots+k_t\geq\sqrt{n-3}-3$, then $n<2^{13}$;
\item if $k_1+k_2+\cdots+k_t\geq{(n-3)/18}$, then $n<2^{10}$;
\item if $k_1+k_2+\cdots+k_t\geq{(n-3)/30}$, then $n<2^{11}$.
\end{enumerate}
\end{lemma}

\begin{proof} We only give here a proof for part~(2).
The other statements
are proved similarly. As $n=2^{k_1}+2^{k_2}+\cdots+2^{k_t}$, we get
$k_1= \lfloor \log_2n\rfloor$ and hence $$k_1+k_2+\cdots+k_t\geq
\lfloor \log_2n\rfloor(\lfloor \log_2n\rfloor+1)/2 \:.$$ However, it is
easy to check that $\sqrt{(n-3)}-3>\lfloor \log_2n\rfloor(\lfloor
\log_2n\rfloor+1)/2$ if $n\geq 2^{14}$. For $2^{13}\leq n<2^{14}$,
the statement follows by direct computations.
\end{proof}

The following lemma is probably known but we include a
short proof for the reader's convenience. It will be needed in the
proof of Theorem~\href{theorem_symmetric}{C}.

\begin{lemma}\label{comparing number of classes of An and Sn} Let
$k(\AAA)$ and $k(\SSS)$ denote the number of conjugacy classes of
$\AAA$ and~$\SSS$, respectively. Then
$$k(\SSS)<2k(\AAA) \:.$$
\end{lemma}

\begin{proof} Let $\Irr_{\text{faithful}}(G)$ and $\Irr_{\text{non-faithful}}(G)$
denote the sets of faithful irreducible characters and non-faithful
irreducible characters, respectively, of a group $G$. Let $a$ and
$b$ be the numbers of self-conjugate partitions and of pairs of
non-self-conjugate partitions, respectively, of $n$. Also, let $c$
and $d$ be the numbers of bar partitions of $n$ with $n-m$ even and
odd, respectively. We have
$$|\Irr_{\text{non-faithful}}(\SSS)|=k(\Sy_n)=a+2b \text{ and }
|\Irr_{\text{faithful}}(\SSS)|=c+2d$$ and
$$|\Irr_{\text{non-faithful}}(\AAA)|=k(\Al_n)=2a+b \text{ and }
|\Irr_{\text{faithful}}(\AAA)|=2c+d \:.$$ Therefore,
$$k(\SSS)=a+2b+c+2d \text{ and } k(\AAA)=2a+b+2c+d$$ and the lemma
follows.
\end{proof}


\section{Complex group algebra of $\AAA$ -
Theorem~\href{theorem_alternating}{A}}\label{section-theoremA}

The aim of this section is to prove
Theorem~\href{theorem_alternating}{A}. Let $G$ be a finite group
such that $\CC G\cong\CC\AAA$. Then $G$ has exactly one linear
character, which is the trivial one, so that $G$ is perfect. Let $M$
be a maximal normal subgroup of $G$. We then have that $G/M$ is
non-abelian simple and moreover
\[\cd^\ast(G/M)\subseteq \cd^\ast(G)=\cd^*(\AAA).\]
To prove the theorem, it is clear that we first have to show
$G/M\cong \Al_n$. We will work towards this aim.

\begin{proposition}\label{proposition_alternating} Let $S$ be a non-abelian simple group such that
$\cd^\ast(S)\subseteq \cd^\ast(\AAA)$. Then $S$ is isomorphic to
$\Al_n$ or
to a simple group of Lie type in even characteristic.
\end{proposition}

\begin{proof} We will eliminate other possibilities for $S$ by using
the classification of finite simple groups. If $5\leq n\leq 9$, then
the set of prime divisors of $S$ is contained in that of $\AAA$,
which in turn is contained in $\{2,3,5,7\}$, hence by using
\cite[Theorem~III]{Huppert-Lempken} and~\cite{Atl1}, the result
follows easily. From now on we assume that $n\geq10$.

\medskip

(i) Alternating groups: Suppose that $S=\Al_m$ with $5\leq m\neq n$.
Since $\cd(A_m)\subseteq\cd(\AAA)$, we get $d_1(\Al_m)\geq
d_1(\AAA)$. As $d_1(\AAA)=n-1\geq 9$ by Lemma~\ref{minimal degrees
of Schur cover of An}, it follows that $d_1(\Al_m)\geq 9$. Thus
$m\geq10$, and so $m-1=d_1(\Al_m)\geq d_1(\AAA)=n-1$. In particular,
we have $m> n$ as $m\neq n$. It follows that $|S|>2|\Al_n|$ and this
violates the hypothesis that $\cd^\ast(S)\subseteq \cd^\ast(\AAA)$.

\medskip

(ii) Simple groups of Lie type in odd characteristic: Suppose that
$S=G(p^k)$, a simple group of Lie type defined over a field of $p^k$
elements with $p$ odd. Since $|S|_p$ is the degree of the Steinberg
character of $S$, we have $|S|_p\in\cd(\AAA)$. As $|S|_p$ is an odd
prime power, Lemma~\ref{prime power degrees of Schur cover of An}
implies that $|S|_p$ must be the degree of a non-faithful character
of $\AAA$. In other words, $|S|_p\in\cd(\Al_n)$. Using
Lemma~\ref{prime power degrees of An}, we deduce that $|S|_p=n-1$.
Hence, $|S|_p=d_1(\AAA)$ is the smallest non-trivial degree of
$\AAA$ by Lemma~\ref{minimal degrees of Schur cover of An}. However,
by~\cite[Lemma~8]{Tong-Viet3} we have $d_1(S)<|S|_p=d_1(\AAA)$,
which is impossible as $\cd(S)\subseteq \cd(\AAA)$.

\medskip

(iii) Sporadic simple groups and the Tits group: Using~\cite{GAP4},
we can assume that $n\geq 14$. To eliminate these cases, observe
that $n\geq \max\{p(S),14\}$, where $p(S)$ is the largest prime
divisor of $|S|$, and that $d_i(S)\geq d_i(\AAA)$ for all $i\geq 1$.
With this lower bound on $n$, we find the lower bounds for
$d_i(\AAA)$ with $1\leq i\leq 7$ using Lemmas~\ref{minimal degrees
of Schur covers} and \ref{minimal degrees of Schur cover of An}.
Choose $i\in\{2,3,\cdots, 7\}$ such that $d_i(\AAA)>d_j(S)$ for some
$j\geq 1$ such that $|i-j|$ is minimal. If $j\geq i$, then we obtain
a contradiction. If $j<i$, then $d_j(S)\in
\{d_k(\AAA)\}_{k=j}^{i-1}$.  Solving these equations for $n$, we
then obtain that either these equations have no solution, or that,
for each solution of $n$, we can find some $k\geq 1$ with
$d_k(\AAA)>d_k(S)$. As an example, assume that $S=O'N$. Then
$n\geq31$ since $p(S)=31$. We have
$d_7(\AAA)=n(n-1)(n-2)(n-7)/24\geq 26970$. As
$d_4(S)=26752<d_7(\AAA)$, it follows that
$d_4(S)\in\{d_4(\AAA),d_5(\AAA),d_6(\AAA)\}$. However, one can check
that these equations have no integer solutions.
\end{proof}

\begin{proposition}\label{proposition-alternating2} Let $S$ be a non-abelian simple group such
that $|S|\mid n!$ and $\cd^\ast(S)\subseteq \cd^\ast(\AAA)$. Then
$S\cong\Al_n$.
\end{proposition}

\begin{proof}
In light of Proposition~\ref{proposition_alternating} and its proof, it remains to assume that
$n\geq9$ and to prove that $S$ cannot be a simple group of Lie type in even characteristic. Assume
to the contrary that $S=G_l(2^k)$, a simple group of Lie type of rank $l$ defined over a field of
$q=2^k$ elements. As above, we then have $|S|_2\in\cd(\AAA)$. By~Lemmas~\ref{prime power degrees of
An} and~\ref{prime power degrees of Schur cover of An}, we have that $|S|_2=n-1$, $|S|_2=2^{\lfloor
(n-2)/2\rfloor}$, or $|S|_2=2^{n/2+\log_2(n-2)-2}$ when $n=2^r+2$. Since the case $|S|_2=n-1$ can
be eliminated as in the proof of the previous proposition, we can assume further that
$$|S|_2=2^{\lfloor (n-2)/2\rfloor} \text{ or } |S|_2=2^{n/2+\log_2(n-2)-2} \text{ when }
n=2^r+2\:.$$ Recalling the hypothesis that $\cd^\ast(S)\subseteq\cd^\ast(\AAA)$, we have
\begin{equation}\label{euqation}|\Irr_{2'}(S)|\leq|\Irr_{2'}(\AAA)|.\end{equation}

\medskip

(i) $S=B_l(2^k)\cong C_l(2^k)$, $D_l(2^k)$, or $\ta D_l(2^k)$. Then
$|S|_2=2^{kl^2}$ or $2^{kl(l-1)}$. In particular,
$$|S|_2\leq 2^{kl^2}\:.$$ As $|S|_2\geq 2^{\lfloor (n-2)/2\rfloor}$, it
follows that
$$kl^2\geq \lfloor (n-2)/2\rfloor\geq (n-3)/2\:.$$ Therefore
$$kl\geq \sqrt{(n-3)/2}\:.$$
Using Lemma~\ref{number of odd degree characters}, we then obtain
$$|\Irr_{2'}(S)|\geq q^l=2^{kl}\geq 2^{\sqrt{(n-3)/2}}\:.$$ Now
Lemma~\ref{number of odd degree characters of Sn} and
inequality~\eqref{euqation} imply
$$k_1+k_2+\cdots+k_t\geq\sqrt{(n-3)/2},$$ where
$n=2^{k_1}+2^{k_2}+\cdots+2^{k_t}$ is the binary expansion of $n$.
Invoking Lemma~\ref{technical lemma for PSL and PSU -Theorem A}(1),
we obtain that $n<2^{15}$.

\medskip

(ii) $S=A_l(2^k)$ or $\ta A_l(2^k)$. Arguing as above, we have
$$k_1+k_2+\cdots+k_t\geq\sqrt{n-3}-3,$$ which forces $n<2^{13}$ by Lemma~\ref{technical lemma for PSL and PSU -Theorem
A}(2).

\medskip

(iii) $S$ is a simple group of exceptional Lie type. Using
Lemma~\ref{technical lemma for PSL and PSU -Theorem A}(3,4), we
deduce that $n<2^{11}$.

\medskip

For each of the above cases, a computer program has checked that
either $|S|$ does not divide $n!$ or $S$ has an irreducible
character degree not belonging to $\cd(\AAA)$ for ``small'' $n$.
This contradiction completes the proof. Let us describe the example
where $S=\Sp_{2l}(2^k)\cong \Omega_{2l+1}(2^k)$. Then we have
\[kl^2=\lfloor (n-2)/2\rfloor \text{ or } {kl^2=n/2+\log_2(n-2)-2} \text{ when }
n=2^r+2.\] Moreover, the condition $|S|\mid n!$ is equivalent to
\[2^{kl^2}\prod_{i=1}^l (2^{2ki}-1) \mid n!.\] By computer computations, we can determine all triples $(k,l,n)$ with
$n<2^{15}$ satisfying the above conditions. It turns out that, for
each such triple, $n$ is at most $170$ and one of the three smallest
character degrees of $S$ is not a character degree of $\AAA$. The
low-degree characters of simple groups of Lie type can be found
in~\cite{TiepZales,Lub,Ng}.
\end{proof}

We are now ready to prove the first main result.

\begin{proof}[\textbf{Proof of Theorem~\href{theorem_alternating}{\textup{A}}}] Recall the
hypothesis that $G$ is a finite group such that $\CC G\cong
\CC\AAA$. Therefore $\cd^\ast(G)=\cd^*(\AAA)$. In particular, we
have $|G|=|\AAA|$ and $G=G'$ since $\AAA$ has only one linear
character. Let $M$ be a maximal normal subgroup of $G$. Then $G/M$
is a non-abelian simple group, say $S$. It follows that
$\cd^\ast(S)=\cd^\ast(G/M)\subseteq \cd^\ast(G)$ and hence
$$\cd^\ast(S)\subseteq \cd^\ast(\AAA)\:.$$ We also have
\[|S|\mid |G|=|\AAA|=n!.\] Applying Propositions~\ref{proposition_alternating} and~\ref{proposition-alternating2}, we
deduce that $S\cong \Al_n$.

We have shown that $G/M\cong \Al_n$. Since $|G|=|\AAA|=2|\Al_n|$, we
obtain $|M|=2$. In particular, $M$ is central in $G$ and therefore
$M\subseteq \bZ(G)\cap G'$. Thus $G\cong \AAA$, as desired.
\end{proof}


\section{Triple and $6$-fold covers of $\Al_6$ and $\Al_7$}\label{sectionA6andA7}

In this section, we aim to prove that every perfect central cover of
$\Al_6$ or $\Al_7$ is uniquely determined up to isomorphism by the
structure of its complex group algebra. To do that, we need the
following result from~\cite[Lemma~2.5]{Nguyen-TongViet}. Here and in
what follows, we write $\Mult(S)$ and $\Schur(S)$ to denote the
Schur multiplier and the Schur covering group (or the Schur cover
for short), respectively, of a simple group $S$.

\begin{lemma}\label{G is uniquely determined by S and |G|} Let $S$ be a
non-abelian simple group different from an alternating group of
degree greater than $13$. Assume that $S$ is different from
$\PSL_3(4)$ and $\PSU_4(3)$. Let $G$ be a perfect group and $M\lhd
G$ such that $G/M\cong S$, $|M|\leq|\Mult(S)|$, and
$\cd(G)\subseteq\cd(\Schur(S))$. Then $G$ is uniquely determined up
to isomorphism by $S$ and the order of $G$.
\end{lemma}

Now we prove the main result of this section.

\begin{theorem}\label{theorem-A6andA7} Let $G$ be a finite group and $H$ a perfect central cover
of $\Al_6$ or $A_7$. Then $G\cong H$ if and only if $\CC G\cong \CC
H$.
\end{theorem}

\begin{proof} First, as $\Al_6\cong \PSL_2(9)$, every perfect central cover of $\Al_6$ can be viewed as
a quasi-simple classical group, which is already studied
in~\cite[Theorem~1.1]{Nguyen}. So it remains to consider the perfect
central covers of $\Al_7$. Let $H$ be one of those and assume that
$G$ is a finite group such that $\CC G\cong\CC H$.

As before, we see that $G$ is perfect and, if $M$ is a normal
maximal subgroup of $G$, we have that $G/M$ is non-abelian simple
and $\cd^\ast(G/M)\subseteq \cd^\ast(H)$. In particular,
\[\cd^\ast(G/M)\subseteq \cd^\ast(\Schur(\Al_7)),\] where
$\Schur(\Al_7)$ is the Schur cover (or the $6$-fold cover) of
$\Al_7$. It follows that $|G/M|\leq 6|\Al_7|=7,560$.
Inspecting~\cite{Atl1}, we come up with
\[G/M\cong \PSL_2(q) \text{ with } 5\leq q \leq 23, \text{ or } \PSL_3(3), \text{ or }\PSU_3(3), \text{ or } \Al_7.\]
Since each possibility for $G/M$ except $\Al_7$ does not satisfy the
inclusion $\cd^\ast(G/M)\subseteq \cd^\ast(\Schur(\Al_7))$, we
deduce that $G/M\cong \Al_7$.

On the other hand, as $\CC G\cong\CC H$, we have $|G|=|H|$. It
follows that $|M|=|\bZ(H)|\leq6$. Using Lemma~\ref{G is uniquely
determined by S and |G|}, we conclude that $G\cong H$.
\end{proof}


\section{Excluding critical character degrees of $\SSS$}\label{section:nondegree}

In this section we prove a non-existence result for special
character degrees of $\SSS$ which will be applied in the next
section. Indeed, with the following proposition we prove a little
more as only the case of even $n$ will be  needed (in fact, the
proof shows that also versions with slightly modified 2-powers can
be obtained).
\begin{proposition}\label{prop:nondegree}
Let $n\in \NN$. If  $2^{[\frac{n-2}2]} (n-1)$ is a character degree
of $\SSS$, then $n\le 8$ and the degree is an ordinary degree
$f_\la$ for $\la\in \{(2),(2,1),(4,2^2)\}$ (or their conjugates), or
the spin degree $g_\mu$ for $\mu=(4,2)$.
\end{proposition}

The strategy for the proof is inspired by the methods used in
\cite{Balog-Bessenrodt-Olsson-Ono,Bessenrodt-Olsson} to classify the
irreducible characters of prime power degrees. A main ingredient is
a number-theoretic result which is a variation of
\cite[Theorem~3.1]{Balog-Bessenrodt-Olsson-Ono}.

\smallskip

First, we define $M(n)$ to be the set of pairs of finite sequences
of integers
$s_1<s_2<\cdots<s_r\le n$, $t_1<t_2<\cdots<t_r\le n$,
with all numbers different from $n-1$,  that satisfy
\begin{itemize}
\item[(i)] $s_i<t_i$ for all $i$;
\item[(ii)] $s_1$ and $t_1$ are primes $>\frac n2$;
\item[(iii)] For $1\le i\le r-1$, $s_{i+1}$ and $t_{i+1}$ contain
prime factors exceeding $2n-s_i-t_i$ and not dividing~$n-1$.
\end{itemize}
We then set
$t(n):=\max\{t_r\mid ((s_i)_{i=1,...,r},(t_i)_{i=1,...,r})\in M(n)\}$,
or $t(n)=0$ when $M(n)=\emptyset$.
Note that for all $n\ge 15$, there are at
least two primes $p,q\neq n-1$ with $\frac n2<p<q\le n$
(e.g., use \cite{Harborth-Kemnitz}); hence for
all $n\ge 15$ the set $M(n)$ is not empty.

\begin{theorem}\label{thm:bound}
Let $n\in \NN$.
Then $n-t(n) \le 225$.

For $15\le n\le 10^9$, we have the tighter bounds
$$n-t(n)\left\{
\begin{array}{cl}
=7 & \text{for } n\in \{30,54\},\\
=5 & \text{for } n\in \{18, 24, 28, 52, 102, 128, 224\},\\
\le 4 & \text{otherwise}.
\end{array}
\right.$$
\end{theorem}
\begin{proof}
For $n> 3.9\cdot 10^8$, the proof follows
the lines of the arguments  for \cite[Theorem 3.1]{Balog-Bessenrodt-Olsson-Ono},
noticing that in the construction given there
the numbers in the sequences are below $n-1$ and
that the chosen prime  factors do not divide $n-1$;
this then gives $n-t(n)\le 225$.

A computer calculation (with Maple) shows that, for all $n\leq
10^9$, we have the claimed bounds and values for $n-t(n)$.
\end{proof}

For any partition $\lambda$ of $n$, we denote by $l(\la)$ the length
of $\la$, and we let $l_1(\la)$ be the multiplicity of~1 in $\la$.
We put $h_i=h(i,1)$ for $1\le i\le l(\la)$; these are the {\em
first-column hook lengths} of $\la$. We set
$\fch(\la)=\{h_1,\ldots,h_{l(\la)}\}$.

First we want  to show that Proposition~\ref{prop:nondegree}
holds for ordinary characters.
Via computer calculations, the claim is easily checked up to $n=44$,
and in particular, we find the stated exceptions for $n<9$.
Thus we have to show that
$$(*) \qquad f_\lambda = 2^{[\frac{n-2}2]} (n-1) $$
cannot hold  for $n\geq 9$; if necessary, we may even assume that $n>44$.

\medskip

To employ Theorem~\ref{thm:bound}, we need some preparations that
are similar to corresponding results in \cite{Balog-Bessenrodt-Olsson-Ono}.
\begin{proposition}\label{primefch}
If $q$ is a prime with $n-l_1(\la)\le q\le n$ and $q\nmid f_{\la}$, then
$$q,\ 2q,\ \ldots,\ \left[\frac nq\right]q \in \fch(\la)\:.$$
\end{proposition}

\begin{proof}
Put $w=\big[\frac nq\big]$, $n=wq+r$, $0\le
r<q$. By assumption, we have
$(w-1)q\le(w-1)q+r=n-q\le l_1:=l_1(\la)$.
Thus $q,2q,\ldots,(w-1)q \in \fch(\la)$.
If $wq\le l_1$, then we are done. Assume that $l_1<wq$.
At most $w$ hooks in $\la$ are of
lengths divisible by $q$ (see e.g.\
\cite[Prop.(3.6)]{Ol}).
If there are only the above $w-1$ hooks
in the first column of length divisible by $q$, then $q|f_\lambda$
since $\prod_{i=1}^w(iq)\mid n!$, a contradiction.
Let
$h(i,j)$ be the additional hook length divisible by $q$. Since
$\lambda\ne(1^n)$, $l_1\le h_2$. If $h_2>l_1$, then
$h(i,j)+h(2,1)>q+l_1\ge n$. By \cite[Cor.~2.8]{Balog-Bessenrodt-Olsson-Ono} we get $j=1$. If
$h_2=l_1$, then $\lambda=(n-l_1,1^{l_1})$ and since $l_1<wq$ there
has to be a hook of length divisible by $q$ in the first row. Since
$n-l_1\le q$ we must have $h_1=wq$.
\end{proof}

In analogy to \cite[Cor.~2.10]{Balog-Bessenrodt-Olsson-Ono}, we deduce
\begin{corollary}\label{cor-fch}
Let $1 \le i < j \le l(\la)$. If\, $h\le n$ has a prime
divisor $q$ satisfying $2n-h_i-h_j<q$ and $q\nmid f_{\la}$,
then $h \in \fch(\la)$.
\end{corollary}

We now combine these results with Theorem~\ref{thm:bound},
similarly as in \cite{Balog-Bessenrodt-Olsson-Ono};
as stated earlier we may assume that $n\ge 15$, and hence there
are least two primes $p,q\neq n-1$ with $\frac n2<p<q\le n$.
Assuming $(*)$ for $\la$,
the hook formula implies that
there have to be hooks of length $p$ and $q$
in $\lambda$.
As argued in \cite{Balog-Bessenrodt-Olsson-Ono}, we then
have $p,q\in\fch(\la)$ or $p,q\in\fch(\overline\la)$;
w.l.o.g., we may assume $p,q\in\fch(\la)$.
Then the assumption~$(*)$ forces {\em any}
prime between $\frac n2$ and $n$, except $n-1$ if this is prime,
to be in $\fch(\la)$. This gives an indication towards the connection
with the sequences belonging to the pairs in $M(n)$.
\smallskip

Indeed, we have the following proposition which is
proved similarly as the corresponding result in~\cite{Balog-Bessenrodt-Olsson-Ono}.
\begin{proposition}\label{squeeze}
Let $n\ge 15$.
Let $((s_i)_{i=1..r},(t_i)_{i=1..r})\in M(n)$.
Let $\la$ be a partition of $n$ such that~$(*)$ holds.
Then
$\{s_1,...,s_r,t_1,...,t_r\}\subset \fch(\la)$
or $\{s_1,...,s_r,t_1,...,t_r\}\subset \fch(\overline\la)$ .

In particular,
$n-h_1\le 225$, and we have tighter bounds for
$n-h_1$ when $n\le 10^9$ as given in Theorem~\ref{thm:bound}.
\end{proposition}

Now we can embark on the {\bf first part of the proof of
Proposition~\ref{prop:nondegree}}, showing the non-existence of
ordinary irreducible characters of the critical degree for $n\ge 9$.
As remarked before, we may assume $n\ge 44$.

Set $m=[\frac{n-2}2]$,
and assume that the partition $\la$ of $n$ satisfies
$$(*) \qquad f_{\la}=2^m(n-1)\:.$$
Let $c=n-h_1$.
By \cite[Prop.~4.1]{Balog-Bessenrodt-Olsson-Ono} we have the following bound for the 2-part of
the degree:
$$(f_\la)_2 \leq n^2 \cdot ((2c+2)!)_2\:.$$
By Proposition~\ref{squeeze}, we have $c\le 225$, and hence
$((2c+2)!)_2 \le (452!)_2=2^{448}$.
Thus
$$2^m\le 2^{448} n^2 \leq 2^{448} (2m+3)^2 \:. $$
A short computation gives $m\le 467$, and hence $n\le 937$.
By Proposition~\ref{squeeze},
$c\le 5$, unless $n=54$, where we only get $c\le 7$.
But for $n=54$ we can argue as follows.
As $\la$ satisfies $(*)$,
w.l.o.g.\ $43,47 \in \fch(\la)$.
Then $l_1(\la) \ge 35$ (by \cite[Prop.~2.6]{Balog-Bessenrodt-Olsson-Ono}), and
hence $17,34\in \fch(\la)$; since $54>3\cdot 17$,
by the hook formula
there has to be one more hook of length divisible by~17 in~$\la$.
As $c\le 7$, this is in the first row or column;
if it is not in the first column, we get a contradiction
considering this hook and the one of length~43.
Thus $51\in \fch(\la)$, and hence $c\le 3$.

Hence for all $n\le 937$ we have
$((2c+2)!)_2 \le (12!)_2=2^{10}$, and
$$2^m\le 2^{10} n^2 \leq 2^{10} (2m+3)^2 \:. $$
This implies $m\le 20$, and hence $n\le 43$,
where the assertion was checked directly.
\qed

\bigskip

Next we deal with the spin characters of $\SSS$. Recall that for a
bar partition $\la$ of~$n$, $\bar g_{\la}$ is the number of shifted
standard tableaux of shape $\la$, and the spin character degree
associated to $\la$ is $g_\la=2^{[\frac{n-l(\la)}2]}\bar g_{\la}$.
Hence the condition on the spin degree translates into the condition
$(\dagger)$ on $\bar g_\la$ given below.
\begin{proposition}\label{prop:gnonexist}
Let $\la$ be a bar partition of~$n$. Then
$$(\dagger )\qquad  \bar g_\lambda = \left\{
\begin{array}{cl}
2^{\big\lceil\frac{l(\la)-2}2\big\rceil}(n-1) & \text{for $n$ even}\\
2^{\big\lceil\frac{l(\la)-3}2\big\rceil}(n-1) & \text{for $n$  odd}
\end{array}
\right.$$ only if $n\le 6$, and $\la=(2)$, $\la=(4,2)$.
\end{proposition}

We note that, for $n\le 34$, the assertion is easily checked by
computer calculation (using John Stembridge's Maple package QF), so
we may assume that $n>34$ when needed.

\medskip

We put $b_i=b(1,i)$ for the {\em first row bar lengths} of~$\la$,
and $\frb(\la)$ for the set of first row bar lengths of~$\la$
(see \cite{Ol} for details on the combinatorics of bars).

In analogy to the case of ordinary characters where we have modified
the results in \cite{Balog-Bessenrodt-Olsson-Ono}, we adapt the
results in \cite{Bessenrodt-Olsson} for the case under consideration
now. Similarly to Proposition~\ref{primefch} before, we have a
version of \cite[Prop.~2.5]{Bessenrodt-Olsson} where, instead of the
prime power condition for $\bar g_\la$, the condition $q\nmid \bar
g_{\la}$ is assumed for the prime $q$ under consideration. For the
corresponding variant of \cite[Lemma~2.6]{Bessenrodt-Olsson} that
says that any prime $q$ with $\frac n2 < q \le n$ and $q \nmid \bar
g_\la$ is a first row bar length of~$\la$, we need two primes
$p_1,p_2 \ne n-1$ with $p_1+p_2-n > \frac n2$. For $n\geq 33$, $n\ne
42$, we always find two primes $p_1,p_2\ne n-1$ such that  $\frac34
n<p_1<p_2\le n$. But for $n=42$, the primes $p_1=31$ and $p_2=37$
are big enough to have $p_1+p_2-n > \frac n2$.
\smallskip

The largest bar length of $\la$ is $b_1=b(1,1)=\la_1+\la_2$.
As before, the preparatory results just described together
with our arithmetical Theorem~\ref{thm:bound}
show that $n-b_1$ is small for a bar partition
$\lambda$ satisfying $(\dagger)$.
More precisely, we obtain:
\begin{proposition}\label{barsqueeze}
Let $n\ge 15$.
Let $((s_1,\ldots,s_r),(t_1,\ldots,t_r))\in M(n)$.
Assume that $\la$ is a bar partition of $n$ that
satisfies~$(\dagger )$.
Then
$s_1,\ldots,s_r,t_1,\ldots,t_r \in \frb(\la)$.

In particular, if $\la$  satisfies~$(\dagger )$, then
$n-b_1\le 225$, and we have tighter bounds for
$n-b_1$ when $n\le 10^9$ as given in Theorem~\ref{thm:bound}.
\end{proposition}

Now we can get into the {\bf second part of the proof of
Proposition~\ref{prop:nondegree}}, showing the non-existence of spin
irreducible characters of the critical degree for $n\ge 7$. We have
already seen that it suffices to prove
Proposition~\ref{prop:gnonexist}, and that we may assume $n\ge 15$.
Set
$$r=\left\{
\begin{array}{cl}
\big\lceil\frac{l(\la)-2}2\big\rceil&  \text{for $n$ even}
\\
\-
\big\lceil\frac{l(\la)-3}2\big\rceil & \text{for $n$ odd}
\end{array}
\right.
$$
and assume that
$\la$ is a bar partition of $n$ that satisfies $(\dagger)$.

Let $c=n-b_1$.
As seen above, we have $c\le 225$, and hence
$l(\la)\le 23$.
Thus $r\le 11$ in any case,
and hence $\bar g_\la \le 2^{11}(n-1)$.

Now, by \cite[Prop. 2.2]{Bessenrodt-Olsson} we know that $\bar g_\la
\geq \frac12 (n-1)(n-4)$ unless we have one of the following
situations: $\la=(n)$and $\bar g_\la=1$, or $\la=(n-1,1)$ and $\bar
g_\la=n-2$. None of these exceptional cases is relevant here, and
thus we obtain $n-4\le 2^{12}$. But for $n\le 4100$ we already know
that $c\le 7$. Then $l(\la)\le 6$ and $r\le 2$, and hence
$$\frac12 (n-1)(n-4) \le \bar g_{\la} \le 4(n-1)\:.$$
But then $n-4\le 8$, a contradiction.
Thus we have now completed the proof of Proposition~\ref{prop:nondegree}.
\qed


\section{Eliminating simple groups of Lie type in even characteristic}\label{section eliminating even charactristic}

Let $G$ be a finite group whose complex group algebra is isomorphic
to that of $\SSS$. In order to show that $G$ is isomorphic to one of
the two double covers of $\Sy_n$, one has to eliminate the
involvement of all non-abelian simple groups other than $\Al_n$ in
the structure of $G$. The most difficult case turns out to be the
simple groups of Lie type in even characteristic.

For the purpose of the next lemma, let $\mathcal{C}$ be the set
consisting of the following simple groups: \[\{
{}^2F_4(2)',\PSL_4(2),\PSL_3(4),\PSU_4(2),\PSU_6(2),\]
\[\mathrm{P}\Omega_8^+(2),\PSp_6(2),{}^2B_2(8),G_2(4),{}^2E_6(2)\}.\]

\begin{lemma}\label{minproj} If $S$ is a simple group of Lie type in characteristic
$2$ such that $|S|_2\geq 2^4$ and $S\not\in\mathcal{C}$, then
$|S|_2<2^{(e(S)-1)/2}$, where $e(S)$ is the smallest non-trivial
degree of an irreducible projective representation of $S$.
\end{lemma}

\begin{proof} Assume that $S$ is defined over a finite field of size $q=2^f$.  If $|S|_2=q^{N(S)}$, then the inequality in the lemma is equivalent to
\begin{equation}\label{eq} e(S)>2N(S)f+1.
\end{equation}

The values of $e(S)$ are available in \cite[Table~II]{TiepZales} for
classical groups and in~\cite{Lub} for exceptional groups. The
arguments for simple classical groups are quite similar. So let us
consider the linear groups. Assume that $S=\PSL_m(q)$ with $m\geq
2$.  We have $N(S)=m(m-1)/2$.  First we assume that $m=2$.  As
$|S|_2=q\geq 2^4$, we deduce that $e(S)=q-1$ and $f\geq 4$.  In this
case, we obtain $e(S)-1=2^f-2$ and $2N(S)f+1=2f+1$.  As $f\geq 4$, we
see that (\ref{eq}) holds. Next we assume that $m\geq 3$ and $S\neq
\PSL_3(4),\PSL_4(2)$.  As $|S|_2\geq 2^4$, we deduce that $S\neq
\PSL_3(2)$.  We have $e(S)={(q^m-q)}/({q-1})$.  Now (\ref{eq}) is
equivalent to
\[\frac{q^m-q}{q-1}>m(m-1)f+1.\] It is routine to check that this
inequality holds for any $m\geq 3$ and $q\geq 2$.

The arguments for exceptional Lie type groups are also similar. For
instance, if $S={}^2B_2(2^{2m+1})$ with $m\geq 1$ then
$|S|_2=2^{2(2m+1)}$ and $e(S)=2^m(2^{2m+1}-1)$.  The inequality can
now be easily checked.
\end{proof}


\begin{proposition}\label{symmetric2bis} Let $G$ be a finite group and $S$ a simple group
of Lie type in characteristic $2$. Suppose that $M\subseteq G'$ is a
normal subgroup of $G$ such that $G'/M\cong S$ and $|G:G'|=2$. Then
$\cd^\ast(G)\neq\cd^\ast(\SSS)$ for every integer $n\geq 10$.
\end{proposition}

\begin{proof} By way of contradiction, assume that
$\cd^\ast(G)=\cd^\ast(\SSS)$ for some $n\geq 10$.  Let $\St_S$ be the
Steinberg character of $G'/M\cong S$.  As $\St_S$ extends to $G/M$
and $|G/M:G'/M|=2$, by Gallagher's Theorem
(see~\cite[Corollary~6.17]{Isaacs} for instance), $G/M$ has two
irreducible characters of degree $\St_S(1)=|S|_2$.  As $n\geq 10$, Lemma~\ref{minimal degrees of Schur cover of An}(1) yields that
$d_1(\SSS)=n-1$.

We claim that $|S|_2>d_1(\SSS)=n-1$.  Suppose for a contradiction
that $|S|_2=d_1(\SSS)$.  If $G/M\cong S\times C_2$, then
$\cd(G/M)=\cd(S)\subseteq \cd(\SSS)$, which implies that $d_1(S)\geq
d_1(\SSS)=|S|_2$.  However, this is impossible since $S$ always has
a non-trivial character degree smaller than $|S|_2$
(see~\cite[Lemma~8]{Tong-Viet3} for instance). Now assume that $G/M$
is almost simple with socle $S$.  If $S\not\cong\PSL_2(q)$ with
$q\geq 4$, then $d_1(G/M)<|S|_2=d_1(\SSS)$ by
\cite[Lemma~2.4]{Tong-Viet2}, which leads to a contradiction as
before since $\cd(G/N)\subseteq \cd(\SSS)$.  Therefore, assume that
$S=\PSL_2(q)$ with $q=2^f\geq 4$.  Then $q=|S|_2=n-1\geq 9$.  If
$q\equiv -1$ (mod $3$), then $d_1(G)=q-1<q=d_1(\SSS)$ by
\cite[Lemma~2.5]{Tong-Viet2}, which is impossible. Hence, $q\equiv
1$ (mod $3$) and $q+1\in\cd(G/M)\subseteq \cd(\SSS)$.  It follows
that $q+1=(n-1)+1=n\in\cd(\SSS)$.  If $n\geq 12$, then
$d_2(\SSS)>2n>n>d_1(\SSS)$, hence $n$ is not a degree of $\SSS$.
Thus $10\leq n\leq 11$.  However, we see that $n-1$ is not a power
of $2$ in either case. The claim is proved.

Assume that $n=2k+1\geq 11$ is odd. By Lemmas \ref{prime power
degrees of Sn} and \ref{prime power degrees of Schur cover of Sn},
$|S|_2=2^k$ is the degree of the basic spin character of $\SSS$.
However, by~\cite[Table~I]{Wales} such a degree has multiplicity
one, which contradicts the fact proved above that $G$ has at least
two irreducible characters of degree $|S|_2$.

Assume $n=2k\geq 10$ is even. By Lemmas~\ref{prime power degrees of
Sn} and~\ref{prime power degrees of Schur cover of Sn} and
\cite[Table~I]{Wales}, $\SSS$ always has the character degree
$2^{k-1}$ with multiplicity $2;$ and if $n=2^r+2$, then it has the
character degree $2^{k-1}(n-2)=2^{2^{r-1}+r}$ with multiplicity $1$.
These are in fact the only non-trivial $2$-power character degrees
of $\SSS$. As in the previous case, by comparing the multiplicity,
we see that $|S|_2\neq 2^{2^{r-1}+r}$.  Thus $|S|_2=2^{k-1}$ is the
degree of the basic spin character of $\SSS$ with multiplicity $2$.
Notice that $k\geq 5$ and hence $|S|_2=2^{k-1}\geq 2^4$.

Now let $\psi\in\Irr(G)$ with $\psi(1)=n-1$.  As $|G:G'|=2$ and
$\psi(1)$ is odd, we deduce that
$\phi=\psi\downarrow_{G'}\in\Irr(G')$ and $\phi(1)=n-1$.  Let
$\theta\in\Irr(M)$ be an irreducible constituent of
$\phi\downarrow_M$.  Then
$\phi\downarrow_M=e(\theta_1+\cdots+\theta_t)$, where
$t=|G':I_{G'}(\theta)|$, and each $\theta_i$ is conjugate to
$\theta\in\Irr(M)$.  If $\theta$ is not $G'$-invariant, then
$\phi(1)=et\theta(1)=n-1\geq \min(S)$, where $\min(S)$ is the
smallest non-trivial index of a maximal subgroup of $S$.  We see
that $\min(S)>d_1(S)\geq e(S)$, where $e(S)$ is the minimal degree
of a projective irreducible representation of $S$, and so $n-1\geq
e(S)$. If $\theta$ is $G'$-invariant and $\phi\downarrow_M=e\theta$
with $e>1$, then $e$ is the degree of a projective irreducible
representation of $S$. It follows that $n-1\geq e\geq e(S)$.  In
both cases, we always have
$$k-1=\frac{n-2}{2}\geq \frac{e(S)-1}{2}\:.$$ Therefore,
\[|S|_2=2^{k-1}\geq 2^{(e(S)-1)/2}.\] By Lemma~\ref{minproj}, we
deduce that $S\in \mathcal{C}$.  Solving the equation
$|S|_2=2^{(n-2)/2}$, we get the degree $n$.  However, by using
\cite{Atl1} and Lemma~\ref{minimal degrees of Schur cover of An}, we
can check that $\cd(G/M)\nsubseteq \cd(\SSS)$ in any of these cases.
For example, assume that $S\cong \ta E_6(2)$.  Then
$|S|_2=2^{36}=2^{(n-2)/2}$, so $n=74$.  By Lemma~\ref{minimal
degrees of Schur cover of An}, we have $d_1(\SSS)=n-1=73$ and
$d_2(\SSS)=n(n-3)/2=2627$.  Using \cite{Atl1}, we know that
$\cd(G/M)$ contains the degree $1938$.  Clearly,
$d_1(\SSS)<1938<d_2(\SSS)$, so $1938\not\in\cd(\SSS)$, hence
$\cd(G/M)\nsubseteq\cd(\SSS)$, a contradiction.

Finally we assume that $et=1$.  Then $\theta$ extends to
$\phi\in\Irr(G')$ and to $\psi\in\Irr(G)$.  Hence
$\phi\downarrow_M=\theta$ and so by Gallagher's Theorem, we have
$\psi\tau\in\Irr(G)$ for every $\tau\in\Irr(G/M)$.  In particular,
\[2^{k-1}(n-1)=\psi(1)|S|_2\in\cd(G)=\cd(\SSS),\] which is impossible by
Proposition~\ref{prop:nondegree}.
\end{proof}


\section{Eliminating simple groups other than $\Al_n$}\label{section eliminating simple groups}

We continue to eliminate the involvement of simple groups other than
$\Al_n$ in the structure of $G$ with $\CC G\cong\CC \SSS$.

\begin{proposition}\label{symmetric2} Let $G$ be an almost simple
group with non-abelian simple socle $S$.  Suppose that
$\cd^\ast(G)\subseteq\cd^\ast(\SSS)$ for some $n\geq 10$.  Then
$S\cong \Al_n$ or $S$ is isomorphic to a simple group of Lie type in
characteristic $2$.
\end{proposition}

\begin{proof} We make use of the classification of finite simple
groups.

\medskip

(i) $S$ is a sporadic simple group or the Tits group. Using
\cite{GAP4}, we can assume that $n\geq 19$.  By Lemma~\ref{minimal
degrees of Schur cover of An}(2), we have $d_2(\SSS)=n(n-3)/2\geq
152$.  Since $d_2(G)\geq d_2(\SSS)\geq 152$, using \cite{Atl1}, we
only need to consider the following simple groups:
\[J_3,Suz,McL,Ru,He,Co_1,Co_2,Co_3,Fi_{22},O'N, HN, Ly, Th, Fi_{23},J_4,Fi_{24}',B,M.\]
To eliminate these groups, we first observe that $n\geq p(S)$, the
largest prime divisor of $|S|$, and $d_i(G)\geq d_i(\SSS)$ for all
$i\geq 1$.  Now with the lower bound $n\geq
\textrm{max}\{19,p(S)\}$, we can find the lower bounds for
$d_i(\SSS)$ with $1\leq i\leq 7$ using Lemmas~\ref{minimal degrees
of Schur covers} and \ref{minimal degrees of Schur cover of An}.
Choose $i\in\{2,3,\cdots, 7\}$ such that $d_i(\SSS)>d_j(G)$ for some
$j\geq 1$ such that $|i-j|$ is minimal. If $i\geq j$, then we obtain
a contradiction. Otherwise, $d_j(G)\in \{d_k(\SSS)\}_{k=j}^{i-1}$.
Solving these equations for $n$, we then obtain that either these
equations have no solution, or that, for each solution of $n$, we
can find some $k\geq 1$ with $d_k(\SSS)>d_k(G)$.

For an example of such a demonstration, assume that $S=O'N$.  In
this case, we have $|{\rm Out}(S)|=2$, so $G=S$ or $G=S.2$.  Since
$p(S)=31$, we have $n\geq 31$.  Assume first that $G=S=O'N$.  Then
$d_4(O'N)=26752$ and, since $n\geq 31$, by Lemma~\ref{minimal
degrees of Schur covers}, $d_7(\SSS)\geq 26970> d_4(O'N)$.  It
follows that $d_4(O'N)\in\{d_i(\SSS)\}_{i=4}^6$.  However, we can
check that these equations are impossible. Now assume $G=O'N.2$.
Then $d_2(G)=26752<26970\leq d_7(\SSS)$ so that $d_2(G)\in
\{d_i(\SSS)\}_{i=2}^6$.  As above,  these equations cannot hold for
any $n\geq 31$.  Thus $\cd(G)\nsubseteq \cd(\SSS)$.

For another example, let $S=M$.  Since $|{\rm Out}(S)|=1$, we have
$G=S$ so that $p(S)=71\in \pi(\SSS)$ and hence $n\geq 71$.  As
$d_1(M)=196883<914480\leq d_7(\SSS)$, we deduce that $d_1(M)\in
\{d_i(\SSS)\;|\;i=1,\cdots,6\}$.  Solving these equations, we obtain
$n=196884$.  But then $d_2(\SSS)>21296,876=d_2(M)$.  Thus
$\cd(M)\nsubseteq \cd(\SSS)$.

\medskip

(ii) $S=\Al_m$ with $m\geq 7$. Note that we consider
$\Al_5\cong\PSL_2(5)$ and $\Al_6\cong\PSL_2(9)$ as groups of Lie
type. Let $\lambda=(m-1,1)$, a partition of $m$. Since $m\geq 7$,
$\lambda$ is not self-conjugate, hence the irreducible character
$\chi_\lambda$ of $\Sy_m$ is still irreducible upon restriction to
$\Al_m$.  Note that ${\rm Aut}(\Al_m)=\Sy_m$ as $m\geq 7$.  Then
$G\in\{\Al_m,\Sy_m\}$ and $G$ has an irreducible character of degree
$m-1$.  Since ${\cd}(G)\subseteq \cd(\SSS)$, we have $m-1\geq
d_1(\SSS)=n-1$, so $m\geq n$.  If $m=n$ then we are done. On the
other hand, if $m>n$ then
\[|G|\geq |S|=|\Al_m|>4|\Al_n|=|\SSS|\] and this violates the
hypothesis $\cd^\ast(G)\subseteq\cd^\ast(\SSS)$.

\medskip

(iii) $S$ is a  simple group of Lie type in odd characteristic.
Suppose that $S=G(p^k)$, a simple group of Lie type defined over a
field of $p^k$ elements with $p$ odd. Let $\St_S$ be the Steinberg
character of $S$. Then, as $\St$ extends to $G$ and
$\St_S(1)=|S|_p$, we have $|S|_p\in\cd(\SSS)$. Using
Lemma~\ref{prime power degrees of Schur cover of Sn}, which says
that all possible prime power degrees of spin characters of $\Sy_n$
are even, we deduce that $|S|_p\in\cd(\Sy_n)$. By Lemma~\ref{prime
power degrees of Sn}, we then obtain that  $|S|_p=n-1$ since $n\geq
10$.  By Lemma~\ref{minimal degrees of Schur cover of An},
$n-1=d_1(\SSS)$ is the smallest non-trivial degree of $\SSS$. Assume
first that $S\neq \mathrm{PSL}_2(q)$. Then $d_1(G)<|S|_p=d_1(\SSS)$
by~\cite[Lemma~2.4]{Tong-Viet2}, which is a contradiction as
$\cd(G)\subseteq \cd(\SSS)$. Now it remains to consider the case
$S=\mathrm{PSL}_2(q)$. We have $q=n-1\geq 9$.  If $G$ has a
character degree which is smaller than $|S|_p=q$, then we obtain a
contradiction as before. So, by~\cite[Lemma~2.5]{Tong-Viet2}, we
have $p\neq3$ and $q\equiv1~(\bmod~3)$ or $p=3$ and
$q\equiv1~(\bmod~4)$.  In both cases, $G$ has an irreducible
character of degree $q+1=n=d_1(\SSS)+1$.  If $n\geq 12$, then
$d_2(\SSS)\geq 2n>n>d_1(\SSS)$ by Lemma~\ref{minimal degrees of
Schur cover of An}, so that $n$ is not a character degree of $\SSS$.
Assume that $10\leq n\leq 11$. Then $n=10$ and $q=9$.  However,
using \cite{Atl1}, we can check that $\cd(G)\nsubseteq \cd(\SSS)$
for every almost simple group $G$ with socle $\PSL_2(9)\cong \Al_6$.
\end{proof}


Combining Propositions~\ref{symmetric2} and~\ref{symmetric2bis}, we
obtain the following results, which will be crucial in the proof of
Theorem~\href{theorem_alternating}{C}.

\begin{proposition}\label{proposition main}
Let $G$ be a finite group and let $M\subseteq G'$ be a normal subgroup of~$G$
such that $G/M$ is an almost simple group with socle $S\neq \Al_n$, where $|G:G'|=2$ and $G'/M\cong S$.  Then
$\cd^\ast(G)\neq\cd^\ast(\SSS)$.
\end{proposition}

\begin{proof} If $n\geq 10$, then the result follows from Propositions~\ref{symmetric2}
and~\ref{symmetric2bis}. It remains to assume that $5\leq n\leq 9$
and suppose by contradiction that $\cd^\ast(G)=\cd^\ast(\SSS)$. Then
$|G|=2n!$ and so $|S|\mid 2n!$, hence $\pi(S)\subseteq
\pi(\SSS)\subseteq \{2,3,5,7\}$.  By
\cite[Theorem~III]{Huppert-Lempken}, one of the following holds:

\begin{enumerate}
  \item If $\pi(S)=\{2,3,5\}$, then $S\cong \Al_5,\Al_6$ or $\PSp_4(3)$.
  \item If $\pi(S)=\{2,3,7\}$, then $S\cong \PSL_2(7),\PSL_2(8)$ or $\PSU_3(3)$.
  \item If $\pi(S)=\{2,3,5,7\}$, then $S\cong \Al_k$  with $7\leq k\leq
  10$, $J_2$, $\PSL_2(49),\PSL_3(4)$, $\PSU_3(5)$, $\PSU_4(3),\PSp_4(7),\PSp_6(2)$ or $\textrm{P}\Omega_8^+(2)$.
\end{enumerate} Now it is routine to check that $\cd(G/M)\nsubseteq\cd(\SSS)$ unless $S\cong\Al_n$, where $G/M$ is almost simple with socle $S$.
\end{proof}

\begin{proposition}\label{symmetric3} Let $G$ be a finite group and let $M\subseteq G'$ be a normal subgroup of~$G$.
Suppose that $\cd^\ast(G)=\cd^\ast(\SSS)$ and $G/M\cong G'/M\times
C_2\cong S^k\times C_2$ for some positive integer $k$ and some
non-abelian simple group $S$. Then $k=1$ and $S\cong \Al_n$.
\end{proposition}

\begin{proof} Since $\cd^\ast(S)\subseteq \cd^\ast(S^k)$, the hypotheses imply that
$$\cd^\ast(S)\subseteq\cd^\ast(\SSS)\:.$$

Assume first that $5\leq n\leq 9$.  Since $|S^k|=|S|^k$ divides
$|\SSS|=2n!$, we deduce that $\pi(S)\subseteq \pi(\SSS)$ and, in
particular, $\pi(S)\subseteq \{2,3,5,7\}$. The possibilities for $S$
are listed in the proof of Proposition~\ref{proposition main} above.
Observe that $|S|$ is always divisible by  a prime $r$ with $r\geq
5$.  Hence, $r^k\mid |\SSS|$, which implies that $k=1$ as $|\SSS|$
divides $2\cdot 9!$.  Now the fact that $S\cong \Al_n$ follows
easily.

From now on we can assume that $n\geq 10$.  Using
Proposition~\ref{symmetric2}, we obtain $S=\Al_n$ or $S$ is a simple
group of Lie type in even characteristic. It suffices to show that
$k=1$ and then the result follows from
Proposition~\ref{symmetric2bis}.

Assume that the latter case holds. Then $S$ is a simple group of Lie
type in characteristic $2$.  By Lemmas~\ref{prime power degrees of
Sn} and \ref{prime power degrees of Schur cover of Sn}, $\SSS$ has
at most two distinct non-trivial $2$-power character degrees, which
are $n-1$ and $2^{\lfloor (n-1)/2\rfloor}$, or $2^{\lfloor
(n-1)/2\rfloor}$ and $2^{2^{r-1}+r}$ with $n=2^r+2$.  By way of
contradiction, assume that $k\geq 2$.  If $k\geq 3$, then $G/M\cong
S^k\times C_2$ has irreducible characters of degrees
\[|S|_2^k>|S|_2^{k-1}>|S|_2^{k-2}>1.\] Obviously, this is impossible
as $\cd(G/M)\subseteq \cd(\SSS)$.  Therefore, $k=2$.  In this case,
$G/M$ has character degree $|S|_2^2$ with multiplicity at least $2$
and $|S|_2$ with multiplicity at least $4$.   It follows that either
$2^{\lfloor (n-1)/2\rfloor}=|S|_2^2$ and $n-1=|S|_2$, or
$2^{2^{r-1}+r}=|S|_2^2$ and $2^{\lfloor (n-1)/2\rfloor}=|S|_2$.
However, both cases are impossible by comparing the multiplicity.

It remains to eliminate the case $S\cong \Al_n$ and $k\geq 2$. By
comparing the orders, we see that
$$2(n!/2)^k|M|=2n!\:.$$ After simplifying, we obtain
\[|M|(n!)^{k-1}=2^k.\] Since $n\geq 10$, we see that, if $k\geq 2$, then the left side is divisible by $5$ while
the right side is not. We conclude that $k=1$  and the proof is now
complete.
\end{proof}


\section{Completion of the proof of Theorem~\href{theorem_symmetric}{C}}\label{section theorem C}

We need one more result before proving Theorem~C.

\begin{proposition}\label{symmetric1} Let $G$ be a finite group and let $S$ be a non-abelian simple group. Suppose
that $|G:G'|=2$ and $G'\cong S^2$ is the unique minimal normal
subgroup of $G$.  Then $\cd(G)\nsubseteq\cd(\SSS)$.
\end{proposition}

\begin{proof} Assume, to the contrary, that ${\cd}(G)\subseteq {\cd}(\SSS)$.
Let $\alpha\in {\rm Irr}(S)$ with $\alpha(1)>1$ and put
$\theta=\alpha\otimes 1\in {\rm Irr}(G')$.  Observe that $\theta$ is
not $G$-invariant, so that $I_G(\theta)=G'$; hence $\theta^G\in {\rm
Irr}(G)$ and so $\theta^G(1)=2\alpha(1)\in \cd(\SSS)$.  On the other
hand, if $\varphi=\alpha\otimes \alpha\in {\rm Irr}(G')$, then
$\varphi$ is $G$-invariant and, since $G/G'$ is cyclic, we deduce
that $\varphi$ extends to $\psi\in {\rm Irr}(G)$, so
$\psi(1)=\alpha(1)^2\in \cd(\SSS)$.  Thus, we conclude that
\begin{equation}\label{equation-proof7.2 1}\text{if } a\in \cd(S)\backslash\{1\} \text{ then }
2a,a^2\in \cd(\SSS).
\end{equation}

Let $r$ be an odd prime divisor of $|S|$. The Ito-Michler theorem then
implies that $r$ divides some character degree, say $a$, of $S$.
Since $a^2\in {\rm cd}(\SSS)$ by \eqref{equation-proof7.2 1}, we
have $r^2\mid 2n!$ and hence  $n\geq 2r$ as $r>2$.  Thus, we have
shown that
\begin{equation}\label{equation-proof7.2}\text{if } r\in\pi(S)-\{2\},  \text{ then } r^2\mid 2n! \text{ and } n\geq
2r.\end{equation}

Using the classification of finite simple groups, we consider the
following cases.

\medskip

(i) $S=\Al_m$, with $m\geq 7$.  As $7\in\pi(S)$, it follows from
\eqref{equation-proof7.2} that $n\geq 14$.  Since $m-1\in \cd(S)$,
both $2(m-1)$ and $(m-1)^2$ are in $\cd(\SSS)$ by
\eqref{equation-proof7.2 1}. As $m\geq 7$,  we also have that
$m(m-3)/2,(m-1)(m-2)/2\in \cd(S)$ and so $m(m-3),(m-1)(m-2)\in
\cd(\SSS)$.  We claim that $m<n$.  Suppose for a contradiction that
$m\geq n$.  As $n\geq 14$, by Lemma~\ref{prime number between n and
d2}, there exists a prime $r$ such that $n/2<r\leq n$.  Hence, the
$r$-part of $2n!$ is just $r$.  However, as $r\leq n\leq m$, $r$
divides $|\Al_m|$ and so $r^2\mid 2n!$ by \eqref{equation-proof7.2},
a contradiction. Thus $m<n$ as claimed.

Since $m\geq 7$, we obtain that
\begin{equation}\label{equation-proof7.2 2}1<2(m-1)<m(m-3)<(m-1)(m-2)<(m-1)^2.\end{equation} By Lemma \ref{minimal
degrees of Schur cover of An}(1), we have  $d_1(\SSS)=n-1$, so
$2(m-1)\geq n-1$ and thus $n\leq 2m-1$.  As $n\geq 14$, we deduce
that $m\geq 8$.

Assume first that $m\in\{8,9,10\}$.  Then $(m-1)^2\in\cd(\SSS)$ is a
prime power. As $3^3\neq (m-1)^2> d_1(\SSS)$, Lemmas \ref{prime
power degrees of Sn} and \ref{prime power degrees of Schur cover of
Sn} yield that $(m-1)^2$ is a power of $2$ and thus $m=9$.  Since
$\{2^3,3^3\}\subseteq \cd(\Al_9)$, we have $\{2^4,2^6,3^6\}\subseteq
\cd(\SSS)$ by \eqref{equation-proof7.2}. As $n\geq 14$, we have
$2^{\lfloor (n-1)/2\rfloor}>2^4$ and $n(n-3)/2>2^4$, so
$d_2(\SSS)>2^4$ by Lemma~\ref{minimal degrees of Schur cover of
An}(2). This forces $2^4=d_1(\SSS)=n-1$ or, equivalently, $n=17$.
But then Lemmas \ref{prime power degrees of Sn} and \ref{prime power
degrees of Schur cover of Sn} yield $3^6=d_1(\SSS)$,  which is
impossible.

Assume next that $m\geq 11$.  Then $n\geq 22$ by
\eqref{equation-proof7.2}. By Lemma \ref{minimal degrees of Schur
cover of An}(2), we have $d_2(\SSS)>2n>2m$. In particular,
$2(m-1)<d_2(\SSS)$. By~\eqref{equation-proof7.2 1}, we have
$2(m-1)\in\cd(\SSS)$, hence $2(m-1)=d_1(\SSS)=n-1$, which implies
that $n=2m-1$.  By Lemma \ref{minimal degrees of Schur cover of
An}(3), we have $d_3(\SSS)=(n-1)(n-2)/2$ and thus by
\eqref{equation-proof7.2 2}, we obtain that
\[(m-1)(m-2)\geq d_3(\SSS)=(n-1)(n-2)/2=(m-1)(2m-3).\] After
simplifying, we have $m-2\geq 2m-3$ or, equivalently, $m\leq 1$, a
contradiction.

\medskip

(ii) $S$ is a finite simple group of Lie type in characteristic $p$,
with $S\neq {}^2F_4(2)'$.  As $|S|$ is always divisible by an odd
prime $r\geq 5$, we have $n\geq 2r\geq 10$ by
\eqref{equation-proof7.2}.  Let $\St_S$ denote the Steinberg
character of $S$.  We can check that $\St_S(1)=|S|_p\geq 4$.  Since
$\St_S(1)\in {\cd}(S)$,  $2\St_S(1)$ and $\St_S(1)^2$ are character
degrees of $\SSS$ by~\eqref{equation-proof7.2 1}. As
$2\St_S(1)<\St_S(1)^2$, we have $\St(1)^2>d_1(\SSS)=n-1$.  Since
$n\geq 10$, Lemmas~\ref{prime power degrees of Sn} and~\ref{prime
power degrees of Schur cover of Sn} yield that $\St_S(1)^2$ is a
$2$-power. Hence, $2\St(1)$ is also a $2$-power. By
\cite[Lemma~8]{Tong-Viet3}, there exists a non-trivial character
degree $x$ of $S$ such that $1<x<\St_S(1)$.  It follows that
$2x<2\St_S(1)$ is also a character degree of $\SSS$.  Therefore,
$2\St_S(1)>d_1(\SSS)=n-1$.  Hence, $\SSS$ has two distinct
non-trivial $2$-power character degrees, neither of which is $n-1$.
It follows that $n=2^r+2\geq 10$ and furthermore
$$2\St_S(1)=2^{\lfloor (n-1)/2\rfloor} \text{ and }
\St_S(1)^2=2^{2^{r-1}+r}$$ by Lemma~\ref{prime power degrees of
Schur cover of Sn}. We write $\St_S(1)=2^N$. Then $2^{r-1}+r=2N$ and
$2^{r-1}=N+1$ since $\lfloor (n-1)/2\rfloor=2^{r-1}$.  Solving these
equations, we have $r=N-1$ and $2^{r-1}=r+2$.  As $n\geq10$, we
deduce that $r\geq 3$.  In this case, it is easy to check that the
equation $2^{r-1}=r+2$ has no integer solution.

\medskip

(iii) $S$ is a sporadic simple group or the Tits group. Since the
arguments are fairly similar, we consider just the case $S=J_3$ as
an example. Recall that $p(S)$ is the largest prime divisor of
$|S|$.
By~\eqref{equation-proof7.2}, we have $n\geq 2p(S)$. 
Since $n\geq 2p(S)\geq 22$, we have $d_2(\SSS)=n(n-3)/2\geq
p(S)(2p(S)-3)$ by applying Lemma \ref{minimal degrees of Schur cover
of An}(3). For $i=1,2$, we have $2d_i(S)\in \cd(G)\subseteq
\cd(\SSS)$ with $1<2d_1(S)<2d_2(S)$. For each possibility for $S$,
we can check using \cite{Atl1} that $p(S)(2p(S)-3)>2d_2(S)$, hence
$d_2(\SSS)>2d_2(S)>2d_1(S)$, which is a contradiction.
\end{proof}


We are now ready to prove the main
Theorem~\href{theorem_symmetric}{C}, which we restate below for the
reader's convenience.

\begin{theoremC}\label{theorem_symmetric-1} Let $n\geq5$. Let $G$ be a finite group
and $\SSS^\pm$ the double covers of $\Sy_n$. Then $\CC G\cong \CC
\SSS^+$ \textup{(}or equivalently $\CC G\cong \CC\SSS^-$\textup{)}
if and only if $G\cong \SSS^+$ or $G\cong\SSS^-$.
\end{theoremC}

\begin{proof} By the hypothesis that $\CC G\cong\CC \SSS$, we have $|G|=2n!$
and as $\SSS$ has two linear characters, we also have $|G:G'|=2$.

First we claim that $G'=G''$. Assume not. Then $H:=G/G'$ is a group
whose commutator subgroup $H'$ is non-trivial abelian of index 2.
Now the induction of any non-principal (linear) character of $H'$ to
$H$ must be irreducible and $2$-dimensional. This is not possible
since $\SSS$ with $n\geq 5$ does not have any irreducible character
of degree 2. Thus $G'=G''$.

As $G'=G''$ and $G'$ is non-trivial, one can choose a normal
subgroup $M$ of $G$ such that $M< G'$ and $$G'/M\cong S^k,$$ where
$S$ is a non-abelian simple group and $S^k$ is a chief factor of
$G$. Let
$$C/M:=\bC_{G/M}(G'/M)\:.$$

(A) First we consider the case $C=M$. Then $G'/M$ is the unique
minimal normal subgroup of $G/M$. Therefore $G/M$ permutes the
direct factors of $G'/M$ (which is isomorphic to $S^k$). It then
follows that $k\leq 2$ as $|G:G'|=2$. Invoking
Proposition~\ref{symmetric1}, we deduce that $k=1$ and thus
$G'/M\cong S$. Therefore, $G'/M$ is the socle of $G/M$. As
$\cd^\ast(G/M)\subseteq\cd^\ast(\SSS)$, Proposition~\ref{proposition
main} then implies that $G'/M\cong \Al_n$. Thus $G/M\cong \Sy_n$ and
also $|M|=2$. In particular, $M$ is central in $G$ and therefore
$M\subseteq \bZ(G)\cap G'$. We conclude that $G$ is one of the two
double covers of $\Sy_n$, as desired.

\medskip

(B) It remains to consider the case $C>M$. Since $C/M
\vartriangleleft G/M$ and $\bZ(G'/M)=1$, it follows that $G'$ is a
proper subgroup of $G'C$. As $|G:G'|=2$, we then deduce that $G=G'C$
and hence
$$G/M=G'/M\times C/M \text{ where } C/M\cong C_2\:.$$
Applying Proposition~\ref{symmetric3}, we obtain that $k=1$ and $S$
is isomorphic to $\Al_n$. In other words, $G'/M\cong \Al_n$. So
$G/M\cong \Al_n\times C_2$. Comparing the orders, we get $|M|=2$ and
so $M\subseteq \bZ(G)$. As $M\leq G'=G''$, it follows that $M\leq
\bZ(G')\cap G''$, which in turn implies that $G'$ is the double
cover of $\Al_n$. We have proved that
\begin{equation}\label{equation2}G'\cong\AAA.\end{equation} Moreover, as $C/M\cong
C_2$ and $|M|=2$, we have \begin{equation}\label{equation3}C\cong
C_4 \text{ or } C_2\times C_2.\end{equation}

Now we claim that $G$ is an (internal) central product of $G'$ and
$C$ with amalgamated central subgroup $M$. To see this, let $x,y\in
G'$ and $c\in C$. Then the facts $C/M=\bC_{G/M}(G'/M)$ and $M\leq
\bZ(G)$ imply
$$[x,y]^c=[x^c,y^c]=[xm_1,ym_2]=[x,y],$$ for some $m_1,m_2\in M$.
Therefore, $C$ centralizes $G'=G''$ and the claim follows. This
claim together with~\eqref{equation2} and~\eqref{equation3} yield
$$4k(\AAA)=4k(G')=k(G'\times C)\leq k(M)k(G)=2k(G),$$ where the inequality
comes from the well-known result that $k(X)\leq k(N)k(X/N)$ for $N$
a normal subgroup of $X$ (see~\cite{Nagao} for instance). Since
$\cd^\ast(G)=\cd^\ast(\SSS)$, we have $k(G)=k(\SSS)$. It follows
that $2k(\AAA)\leq k(\SSS)$. This however contradicts
Lemma~\ref{comparing number of classes of An and Sn} and the theorem
is now completely proved.
\end{proof}


\end{document}